\documentclass{amsproc}
\usepackage{epsfig,amscd,amssymb,amsmath,amsfonts}
\usepackage[margin=1.05in]{geometry}
\usepackage[colorlinks]{hyperref}
\usepackage{enumitem}
\usepackage{xcolor}
\newtheorem{theorem}{Theorem}[section]
\newtheorem{lemma}[theorem]{Lemma}
\newtheorem{proposition}{Proposition}[section]
\theoremstyle{definition}
\newtheorem{definition}[theorem]{Definition}

\newtheorem{example}[theorem]{Example}

\theoremstyle{remark}
\newtheorem{remark}[theorem]{Remark}

\numberwithin{equation}{section}

\begin{document}

\title{A Simplicial Construction for Noncommutative Settings}

\author{Samuel Carolus}
\address{Department of Mathematics and Statistics, Bowling Green State University, Bowling Green, OH 43403}
\email{carolus@bgsu.edu}

\author{Jacob Laubacher }
\address{
Department of Mathematics, St. Norbert College, De Pere, WI 54115 }
\email{jacob.laubacher@snc.edu}

\author{Mihai D. Staic}
\address{Department of Mathematics and Statistics, Bowling Green State University, Bowling Green, OH 43403 }
\address{Institute of Mathematics of the Romanian Academy, PO.BOX 1-764, RO-70700 Bu\-cha\-rest, Romania.}
\email{mstaic@bgsu.edu}



\subjclass[2010]{Primary  16E40, Secondary 	18G30}


\keywords{Higher order Hochschild homology.}

\begin{abstract}
In this paper we present a general construction that can be used to define the higher Hochschild homology for a noncommutative algebra. We also discuss other examples where this construction can be used. 
\end{abstract}

\maketitle

 \section{Introduction} Higher order Hochschild homology, $H^{{\mathbf X_{\bullet}}}_n(A,M)$, was introduced by Pirashvili in \cite{p}. It is associated  to a commutative $k$-algebra $A$, a symmetric bimodule $M$, and a simplicial set ${\mathbf X_{\bullet}}$.  When the simplicial set ${\mathbf X_{\bullet}}$ models $S^1$ with the usual simplicial structure, one recovers the usual Hochschild homology. The cohomology version of this construction was introduced by Ginot in \cite{gi}. 

Secondary (co)homology of a triple $(A,B,\varepsilon)$ was introduced in \cite{sta}. In order for the construction to work we must have that  the morphism $\varepsilon:B\to A$ gives a $B$-algebra structure on $A$, and in particular $B$ must be commutative.  

As noted above, higher order Hochschild  (co)homology is defined only for commutative $k$-algebras, while Hochschild (co)homology is defined for any $k$-algebra.  The problem comes from the fact that for a general simplicial set $({\mathbf X_{\bullet}},d_i,s_i)$ we do not have a natural order on the fibers of the maps $d_i$. This means that there is a choice to be made when we define the pre-simplicial $k$-module corresponding to higher order Hochschild (co)homology. 
One possible approach for this problem is to restrict ourselves to those simplicial sets that do have a natural order on the fibers of $d_i$. However this approach does not provide a lot of new examples since any such simplicial set must be of dimension one (see \cite{b}). 

A similar problem appears when we want to define the secondary (co)homology of a triple $(A,B,\varepsilon)$, and the $k$-algebra $B$ is not commutative. There is a choice to be made when one wants to write the formulas for the simplicial maps, and none of those choices give a simplicial module (unless $B$ is commutative).

In this paper we present a construction that allows us to extend several homological constructions  to noncommutative settings. For this we use the simplicial nature of the higher order Hochschild  (co)homology. 
First, we show that to a so called $\mathbf{\Lambda}$-system we can associate a unique maximal pre-simplicial module. Then we construct several natural examples of $\mathbf{\Lambda}$-systems.  In particular, we associate one such $\mathbf{\Lambda}$-system to a simplicial set $\mathbf{X_{\bullet}}$, a $k$-algebra $A$ and an $A$-bimodule $M$. Then we consider the associated pre-simplicial module and take its homology.  
When $A$ is commutative and $M$ is $A$-symmetric we recover the usual higher Hochschild homology $H_n^{{\mathbf X_{\bullet}}}(A,M)$. Our construction can also be used to define a secondary homology in the noncommutative setting.  

We discuss in detail the case when ${\mathbf X}_\bullet$ is modeled by $S^1$. We show that if $A$ is a commutative $k$-algebra and $M$ is a symmetric $A$-bimodule, then $H^{S^1}_n(A,M)\simeq H^{S^1}_n(M_l(A),M_l(M))$, and therefore we have  Morita invariance for this case. In the last section we give an account of other related problems and some open questions.

\section{Preliminaries} 

In this paper $k$ is a field, $\otimes$ is $\otimes_k$, all maps are $k$-linear, etc. We recall a few facts and definitions that will be useful in the upcoming sections.

We say that $(X_{\bullet}, d_i)$ is a pre-simplicial object  in a category $\mathcal{C}$ if for every $n\in \mathbb{N}$, we have an object $X_n\in \mathcal{C}$, and for all $0\leq i\leq n$ we have morphisms $\delta_i:X_{n+1}\to X_{n}$  that satisfy the following relation:
\begin{eqnarray}
&&\delta_i\delta_j=\delta_{j-1}\delta_i\; {\rm ~if~} \; i<j. \label{pres}
\end{eqnarray}
When $\mathcal{C}$ is the category of vector spaces over $k$, we say that $(X_{\bullet}, d_i)$ is a pre-simplicial $k$-module.

Let $A$ be a $k$-algebra (not necessarily  commutative),  and $M$ be an $A$-bimodule. 
We consider the pre-simplicial module $(C_n(A, M), d_i)$  that is used to define Hochschild homology. That is $C_n(A, M)=M\otimes A^{\otimes n}$ and
\begin{eqnarray}
d_i(x_0\otimes x_1\otimes ...\otimes x_n)= \left\{\begin{array}{ll}
  x_0 x_1\otimes x_2\otimes ...\otimes x_n& \mbox{ if $i=0$}\\ 
  x_0\otimes ...\otimes x_{i-1}\otimes x_ix_{i+1}\otimes x_{i+2}\otimes ...\otimes x_n & \mbox{ if  $1\leq i\leq n-1$ }\\
 x_nx_0\otimes x_1\otimes ...\otimes x_{n-1} & \mbox{ if  $i= n$. }\\
 \end{array}\right.\label{C2}
\end{eqnarray}
For  more results concerning Hochschild (co)homology, we refer to \cite{g1}, \cite{g2}, \cite{h},  and  \cite{lo}.

We recall from  \cite{p} the construction of the higher order Hochschild homology.  Assume that $A$ is a commutative $k$-algebra, and $M$ a symmetric $A$-bimodule.

Let $V$ be a finite pointed set such that $\vert V\vert=v+1$. We define $\mathcal{L}(A,M)(V)=M\otimes A^{\otimes v}$. 
For $\phi:V\to W$   we define 
$$\mathcal{L}(A,M)(\phi): \mathcal{L}(A,M)(V)\to \mathcal{L}(A,M)(W)$$
determined as follows:
$$\mathcal{L}(A,M)(\phi)(a_0\otimes a_1\otimes ...\otimes a_v)=b_0m\otimes b_1\otimes ...\otimes b_{w}$$ 
where $$b_i=\prod_{\{j\in V | \phi (j)=i\}}a_j.$$
Take $\mathbf{X}=(X_{\bullet},d_i, s_i)$ to be a finite pointed simplicial set, and define 
$$C^{X_\bullet}_n(A,M)=\mathcal{L}(A,M)(X_n).$$ For each $d_i:X_{n}\to X_{n-1}$ we take $(d_i)_*=\mathcal{L}(A,M)(d_i):C^{X_{\bullet}}_{n}(A,M)\to C^{X_{\bullet}}_{n-1}(A,M)$ and take $\partial_{n}:C_{n}^{X_{\bullet}}(A,M)\to C_{n-1}^{X_{\bullet}}(A,M)$ defined as $\partial_n=\sum_{i=0}^{n}(-1)^i(d_i)_*$. 

The homology of this complex is denoted by $H_{\bullet}^{\mathbf{X}}(A,M)$ and is called the higher order Hochschild homology. When  $\mathbf{X}_\bullet$ is modeled by $S^1$ with the usual simplicial structure, one recovers the complex that defines Hochschild homology. For more results concerning higher order Hochschild (co)homology we refer to \cite{gi}, \cite{gi2}, \cite{gtz}, and \cite{p}. 

Secondary cohomology was introduced in \cite{sta} in order to study $B$-algebra structures on $A[[t]]$. 
The homology version and the associated cyclic (co)homology were introduced and studied in  \cite{jma}. The relation between the secondary and  higher order Hochschild cohomology was established in  \cite{bm}.

\section{A Simplicial Construction for Noncommutative Settings} 

In this section we give a general construction that is designed  to construct pre-simplicial modules in noncommutative settings. 

Its practical relevance will become apparent in the next section, when we use it to define several (co)homology theories for noncommutative algebras. 
First, we need a few definitions.

\begin{definition} Suppose that for each $n\in \mathbb{N}^*$, and for each $0\leq i\leq n$ we have a finite set $\Lambda_n^i$. We call such a collection a  \textbf{$\Delta$-indexing set}, and we denote it by $\mathbf{\Lambda}=\{\Lambda_n^i\,|\,n\in \mathbb{N}^*,\; i=0,\dots,n\}$.  
\end{definition}

\begin{definition}
Let $\mathbf{\Lambda}=\{\Lambda_n^i\,|\,n\in \mathbb{N}^*,\; i=0,\dots,n\}$ be a $\Delta$-indexing set.  We call $\mathcal{M}=(M_n,d_i^\alpha)$ a \textbf{$\mathbf{\Lambda}$-system} if it consists of a collection of $k$-vector spaces $\{M_n\}_{n=0}^{\infty}$, and a collection of $k$-linear morphisms $d_i^{\alpha}:M_n\rightarrow M_{n-1}$ for all $\alpha \in \Lambda_n^i$.
\end{definition}

Note that if $\vert\Lambda_n^i\vert=1$ for all $n\in \mathbb{N}^*$, and all $0\leq i\leq n$, then a $\mathbf{\Lambda}$-system is a pre-simplicial $k$-module. However, in general, a $\mathbf{\Lambda}$-system  does not automatically define a pre-simplicial $k$-module or a chain complex.  The plan is to prove that every $\mathbf{\Lambda}$-system contains a unique  maximal pre-simplicial $k$-module.

\begin{definition}\label{l-def}
Let $\mathcal{M}=(M_n,d_i^\alpha)$ be a $\mathbf{\Lambda}$-system. We call $A_\bullet=(A_n)_{n\geq0}$ a \textbf{$\lambda$-subcomplex} of the $\mathbf{\Lambda}$-system $\mathcal{M}$ if $A_n$ is a sub-vector-space of $M_n$ for every $n$, and for $0\leq i \leq n$ we have, \\
(i) $d_i^\alpha|_{A_n}=d_i^\beta|_{A_n}$ for all $\alpha,\beta \in \Lambda_n^i$, with this common restriction denoted $d_i^A$,\\
(ii) $d_i^A(A_n)\subseteq A_{n-1}$, \\
(iii) $d_i^Ad_j^A=d_{j-1}^Ad_i^A$ for $i<j$.
\end{definition}

\begin{remark}
Notice that (ii) and (iii) imply that $(A_n, d_i^A)$ is a pre-simplicial module and in particular we get a chain complex  (hence the name $\lambda$-sucomplex).
\end{remark}

\begin{remark}\label{restriction}
Let $\mathcal{M}$ be a $\mathbf{\Lambda}$-system, and $\mathcal{S}$ denote the collection of all $\lambda$-subcomplexes. It is obvious that  $\mathcal{S}\neq \emptyset$  We impart a partial order on $\mathcal{S}$ by saying $A_\bullet \leq B_\bullet$ if there exists inclusions $A_n \subseteq B_n$ in every dimension $n$.  Notice that both $d_i^A=d_i^\alpha |_{A_n}$ and $d_i^B=d_i^\alpha |_{B_n}$, so since $A_n\subseteq B_n$, we have $d_i^B|_{A_n}=d_i^A$. 
\end{remark}

\begin{theorem}\label{max}
Let $\mathcal{M}$ be a $\mathbf{\Lambda}$-system.  Then $\mathcal{M}$ has a unique maximal $\lambda$-subcomplex. In particular we have a unique maximal pre-simplicial $k$-module $\Theta(\mathcal{M})$. 
\end{theorem}
\begin{proof}
First, we show the existence of a maximal $\lambda$-subcomplex.  To do this, we shall use Zorn's Lemma.  Consider a countable, totally ordered subset of $\mathcal{S}$, i.e. $\{A_\bullet^m\}_{m=0}^\infty$ with $A_\bullet^m \leq A_\bullet^{m+1}$.  Claim: $\Omega_\bullet :=\bigcup_{m=1}^{\infty} A_\bullet ^m$ (where $\Omega_n=\bigcup\limits_{m\geq 1}A_n^m$) is a $\lambda$-subcomplex.\\
(i): Indeed, if $a\in \Omega_n$, then $a\in A_n^m$ for some $m$. Then for all $\alpha,\beta \in \Lambda_n^i$, $d_i^\alpha(a)=d_i^\beta(a)$ since $A_\bullet^m$ is a $\lambda$-subcomplex. Thus (i) is satisfied.\\
(ii): Similarly, as $a\in \Omega_n$ is also in some $A_n^m$, $d_i^\Omega(a)=d_i^{A^m}(a)\in A_{n-1}^m\subseteq \Omega_{n-1}$.  Hence (ii) is satisfied.\\
(iii): As above, take $a\in \Omega_n$.  Then $a\in A_n^m$ for some $m$, so $$d_i^\Omega(d_j^\Omega(a))=d_i^{A_m}(d_j^{A_m}(a))=d_{j-1}^{A_m}(d_i^{A_m}(a))=d_{j-1}^\Omega(d_i^\Omega(a)),$$  therefore (iii) holds.\\
Thus, $\Omega_\bullet$ is a $\lambda$-subcomplex (i.e. $\Omega_\bullet\in \mathcal{S}$).\\
 Now, being the union of all $A_\bullet^m$, each $A_\bullet^m \leq \Omega_\bullet$, so this is indeed an upper bound of the totally ordered subset $\{A_\bullet^m\}_{m=0}^\infty$.  Thus, by Zorn's Lemma, there exists a maximal element of $\mathcal{S}$.

Now we show there is a unique maximal $\lambda$-subcomplex.  Suppose there are two maximal $\lambda$-subcomplexes, $C_\bullet$ and $D_\bullet$.  Consider $Y_\bullet := C_\bullet + D_\bullet$, where $Y_n$ as a $k$-vector space is $C_n + D_n=\{y\in M_n~|~y=c+d, \text{~for some~} c\in C_n,~d\in D_n\}$. We show that $Y_\bullet$ is a $\lambda$-subcomplex.\\
(i): Take $y\in Y_n$.  Then $y=c+d$ for some $c\in C_n$ and $d\in D_n$.  So for all $\alpha, \beta \in \Lambda_n^i$, we have  $$d_i^\alpha(y)=d_i^\alpha(c+d)=d_i^\alpha(c)+d_i^\alpha(d)=d_i^\beta(c)+d_i^\beta(d)=d_i^\beta(c+d)=d_i^\beta(y).$$ This shows (i).\\
(ii): If $y\in Y_n$, then $y=c+d$ for some $c\in C_n$ and $d\in D_n$, so $$d_i^Y(y)=d_i^Y(c+d)=d_i^Y(c)+d_i^Y(d)=d_i^C(c)+d_i^D(d)\in C_{n-1}+D_{n-1}=Y_{n-1}.$$ Hence (ii) holds.\\
(iii): Let $i<j$, and take  $y\in Y_{n}$ with $y=c+d$ for some $c\in C_{n}$ and $d\in D_n$.  Since $C_\bullet$ and $D_\bullet$ satisfy (iii) and using the observation in the proof for (ii), we have: $$d_i^Y(d_j^Y(y))=d_i^Y(d_j^Y(c+d))=d_i^Y(d_j^C(c)+d_j^D(d))=d_i^Y(d_j^C(c))+d_i^Y(d_j^D(d))=$$ $$d_i^C(d_j^C(c))+d_i^D(d_j^D(d))=d_{j-1}^C(d_i^C(c))+d_{j-1}^D(d_i^D(d))$$ and on the other hand, $$d_{j-1}^Y(d_i^Y(y))=d_{j-1}^Y(d_i^Y(c+d))=d_{j-1}^Y(d_i^C(c)+d_i^D(d))=d_{j-1}^Y(d_i^C(c))+d_{j-1}^Y(d_i^D(d))=$$ $$d_{j-1}^C(d_i^C(c))+d_{j-1}^D(d_i^D(d)).$$  
Thus, $d_i^Yd_j^Y=d_{j-1}^Yd_i^Y$, so (iii) is satisfied. 

Therefore, $Y_\bullet$ is a $\lambda$-subcomplex.

Clearly there are injections $C_n \hookrightarrow Y_n$ and $D_n \hookrightarrow Y_n$, but they were chosen to be maximal, so it must be that $C_\bullet=Y_\bullet=D_\bullet$.  Hence, a maximal $\lambda$-subcomplex is unique, and we denote it by $\Theta(\mathcal{M})$.
\end{proof}

\begin{definition}
Let $\mathcal{M}$ be a $\mathbf{\Lambda}$-system with the unique maximal $\lambda$-subcomplex $\Theta(\mathcal{M})$.  We call the homology of $\Theta(\mathcal{M})$  the  {\bf $\mathbf{\Lambda}$-homology group of $\mathcal{M}$}, and we denote it by  $H_n(\mathcal{M}):=H_n(\Theta(\mathcal{M})_\bullet)$.  
\end{definition}

Next we need to talk about morphisms between $\mathbf{\Lambda}$-systems.

\begin{definition} Take $\mathbf{\Gamma}$ and $\mathbf{\Lambda}$ to be two $\Delta$-indexing sets, $\mathcal{M}=(M_n,d_i^\beta)$  a $\mathbf{\Gamma}$-system, and $\mathcal{N}=(N_n,\delta_i^{\alpha})$ a $\mathbf{\Lambda}$-system. A \textbf{$\lambda$-morphism} from $\mathcal{M}$ to $\mathcal{N}$ is a collection of $k$-linear maps $f_n:M_n\to N_n$ for all $n\in \mathbb{N}$, such that if $n\geq 1$, then for all $0\leq i\leq n$, and all $\alpha\in \Lambda_n^i$ there exists a $\beta\in \Gamma_n^i$ such that  $\delta_i^{\alpha}f_n=f_{n-1}d_i^{\beta}$. 
 \label{def81}
\end{definition} 

We have the following result.  
\begin{lemma} Take $\mathbf{\Gamma}$ and $\mathbf{\Lambda}$ to be two $\Delta$-indexing sets, $\mathcal{M}=(M_n,d_i^\beta)$  a $\mathbf{\Gamma}$-system, and $\mathcal{N}=(N_n,\delta_i^{\alpha})$ a $\mathbf{\Lambda}$-system. If $f:\mathcal{M}\to \mathcal{N}$ is a $\lambda$-morphism then $f$ induces a morphism of pre-simplicial modules $f:\Theta(\mathcal{M})\to \Theta(\mathcal{N})$. 
 \label{lemma10}
\end{lemma}

\begin{proof} First we show that $f_n(\Theta(\mathcal{M})_n)\subseteq \Theta(\mathcal{N})_n$. If $n=0$ that is obvious since $\Theta(\mathcal{N})_0=N_0$. Define  $A_s\subseteq N_s$ determined by $$A_s=f_s( \Theta(\mathcal{M})_{s}), \; {\rm for}\; {\rm all} \; s\geq 0.$$ 
We want to show that $A_{\bullet}=(A_s)_{s\geq 0}$ defines a $\lambda$-subcomplex in $\mathcal{N}$.

Because $(\Theta(\mathcal{M})_{s})_{s\geq 0}$ is a $\lambda$-subcomplex of $\mathcal{M}$, then for all $\beta_1$, $\beta_2$ in $\Gamma_s^i$ we have that $d_i^{\beta_1}=d_i^{\beta_2}$ on $\Theta(\mathcal{M})_{s}$. We will denote this map by $d_i$ (suppressing the $s$ index). 
	
Take $n\geq 1$ and $0\leq i\leq n$. Since $f$ is a $\lambda$-morphism then for  $\alpha_1$, $\alpha_2 \in \Lambda_n^i$  we can find $\beta_1,\beta_2\in \Gamma_n^i$ such that 
$\delta_i^{\alpha_1}f_n=f_{n-1}d_i^{\beta_1}$ and $\delta_i^{\alpha_2}f_n=f_{n-1}d_i^{\beta_2}$.   
Take $x\in A_n$ with $x=f_n(c)$ for some $c\in \Theta(\mathcal{M})_{n}$.  
We have $$\delta_i^{\alpha_1}(x)=\delta_i^{\alpha_1}(f_n(c))=f_{n-1}(d_i^{\beta_1}(c))=f_{n-1}(d_i(c))=f_{n-1}(d_i^{\beta_2}(c))=\delta_i^{\alpha_2}(f_n(c))=\delta_i^{\alpha_2}(x),$$ 
which means that $\delta_i^{\alpha_1}=\delta_i^{\alpha_2}$ on $A_n$ for all $\alpha_1$, $\alpha_2\in \Lambda_n^i$.  And so we have condition (i) from Definition \ref{l-def}. We denote the common restriction by $\delta_i^A$. 

Take $x\in A_n$ with $x=f_n(c)$ for some $c\in \Theta(\mathcal{M})_{n}$. 
Then we have $$\delta_i^A(x)=\delta_i^Af_n(c)=\delta_i^{\alpha}f_n(c)=f_{n-1}d_i^{\beta}(c)=f_{n-1}d_i(c)\in A_{n-1},$$ 
for some $\alpha\in \Lambda_n^i$ (and the corresponding $\beta\in \Gamma_n^i$). This means that $$\delta_i^A(A_n)\subseteq A_{n-1},$$
and so we have condition (ii) from Definition \ref{l-def}.
 
Finally, for all $i<j$, and $x=f_n(c)$ for some $c\in \Theta(\mathcal{M})_{n}$ we have
$$\delta_i^A\delta_j^A(x)=\delta_i^A\delta_j^A(f_n(c))=f_{n-2}(d_id_j(c))=f_{n-2}(d_{j-1}d_i(c)=\delta_{j-1}^A\delta_i^A(f_n(c))=\delta_{j-1}^A\delta_i^A(x).$$ 
Thus $(A_s)_{s\geq 0}$ defines an $\lambda$-subcomplex, and so we get that $A_n=f_n(\Theta(\mathcal{M})_n)\subseteq \Theta(\mathcal{N})_n$ for all $n\in \mathbb{N}$. 

We already noticed that $\delta_i^Af_n=f_{n-1}d_i$, which means that $f$ is a morphism of pre-simplicial modules from $\Theta(\mathcal{M})$ to $\Theta(\mathcal{N})$. 
\end{proof} 

\section{A Few Examples}

\subsection{Higher Order Hochschild Homology}

Let $A$ be a $k$-algebra (not necessarily commutative), and $M$ be an $A$-bimodule. 
As a warm-up,  we define higher order Hochschild homology over the sphere $S^2$.  We will use the description from \cite{l2} as a point of reference. 

\begin{example} \label{example2}
Set $\Lambda_n^0=\{\rho~|~\rho\in S_{n}\}$, where $S_{n}$ is the set of permutations of $\{0,2,\dots, n\}$. For $0<i<n$, set $$\Lambda_n^i=\{\sigma=(\sigma_1,\cdots,\sigma_{n-1})~|~\sigma_j\in\{1,\tau\}\},$$
and $\Lambda_n^n=\{\rho~|~\rho\in S_{n}\}$, where $S_{n}$ is the set of permutations of $\{0,\dots, n-1\}$. 

Let $\rho$ act on a tensor product of length $n$ by permuting the elements according to $\rho$ and then taking the product (i.e. $\rho(x_0\otimes x_2\otimes ... \otimes x_n)=x_{\rho(0)}x_{\rho(2)}...x_{\rho(n)}$). 
Take $$\sigma_j(x\otimes y)=\begin{cases}xy ~~\text{if}~~ \sigma_j=1\\
yx ~~\text{if}~~ \sigma_j=\tau .\end{cases}$$
We define a $\mathbf{\Lambda}$-system $\mathcal{F}$ by taking $F_n=M\otimes A^{\otimes\frac{n(n-1)}{2}}$, and $d_i^\sigma:M\otimes A^{\otimes\frac{n(n-1)}{2}}\to M\otimes A^{\otimes\frac{(n-1)(n-2)}{2}}$ defined as follows:
$$d_0^\rho(m_0\otimes\begin{pmatrix} 
1 & a_{1,2} & \cdots & a_{1,n}\\
 & \ddots & \vdots & \vdots\\
 & & 1 & a_{n-1,n}\\
 & & & 1\end{pmatrix})=
\rho(m_0\otimes a_{1,2}\otimes\cdots\otimes a_{1,n})\otimes\begin{pmatrix} 
1 & a_{2,3} & \cdots & a_{2,n}\\
& \ddots & \vdots & \vdots\\
& & 1 & a_{n-1,n}\\
& & & 1\end{pmatrix}.$$

For $1\leq i\leq n-1$,

$$d_i^\sigma(m_0\otimes\begin{pmatrix} 
1 & a_{1,2} & \cdots & a_{1,n}\\
 & \ddots & \vdots & \vdots\\
 & & 1 & a_{n-1,n}\\
 & & & 1\end{pmatrix})=$$ 
$$
\sigma_i(m_0\otimes a_{i,i+1}){\otimes\begin{pmatrix} 
1 & a_{1,2} & \cdots & \sigma_1(a_{1,i}\otimes a_{1,i+1}) & \cdots & \cdots & a_{2,n}\\
& \ddots & \ddots & \vdots & \cdots & \vdots & \vdots\\
& & 1 & \sigma_{i-1}(a_{i-1,i}\otimes a_{i-1,i+1}) & \cdots & \cdots & a_{i-1,n}\\
& & & 1 & \sigma_{i+1}(a_{i,i+2}\otimes a_{i+1,i+2}) & \cdots & \sigma_{n-1}(a_{i,n}\otimes a_{i+1,n})\\
& & & & \ddots & \vdots & \vdots\\
& & & & & 1 & a_{n-1,n}\\
& & & & & & 1\end{pmatrix}}.$$

Finally, 
$$d_n^\rho(m_0\otimes\begin{pmatrix} 
1 & a_{1,2} & \cdots & a_{1,n}\\
 & \ddots & \vdots & \vdots\\
 & & 1 & a_{n-1,n}\\
 & & & 1\end{pmatrix})=
\rho(m_0\otimes a_{1,n}\otimes\cdots\ a_{n-1,n})\otimes\begin{pmatrix} 
1 & a_{1,2} & \cdots & a_{1,n-1}\\
& \ddots & \vdots & \vdots\\
& & 1 & a_{n-2,n-1}\\
& & & 1\end{pmatrix}.$$
Notice that when $A$ is commutative and $M$ is $A$-symmetric we get the usual higher order Hochschild homology $H_n^{S^2}(A,M)$.
\end{example}

Next, we want to define higher order Hochschild homology for a general simplicial set. 
\begin{example} \label{example3}  Let $\mathbf{X}=(X_{\bullet},d_i,s_i)$ be a pointed simplicial set.
Consider the  $\Delta$-indexing set $\mathbf{\Lambda}^{X_{\bullet}}$ defined by  $$\Lambda_n^i=\prod_{j\in X_{n-1}}S_{Z_n^i(j)},$$ where  $S_Z$ is the symmetric group on the set $Z$, and for $j\in X_{n-1}$ we set $Z_n^i(j)= d_i^{-1}(j)$  where $d_i:X_{n}\to X_{n-1}$.

Let $A$ be a $k$-algebra and $M$ an $A$-bimodule. We define  the $\mathbf{\Lambda}^{X_{\bullet}}$-system $\mathcal{C}^{X_{\bullet}}(A,M)$ as follows. For each $n$ define $C^{X_\bullet}_n(A,M)=M\otimes A^{\otimes x_n}$ where $x_n=\vert X_n\vert-1$. For $\sigma=(\sigma_0, \sigma_1,...,\sigma_{x_{n-1}})\in \Lambda_n^i$ we define 
$d_i^{\sigma}: C^{X_\bullet}_n(A,M)\to C^{X_\bullet}_{n-1}(A,M)$ determined by
$$d_i^{\sigma}(a_0\otimes a_1\otimes ...\otimes a_{x_n})=b_0^{\sigma}\otimes b_1^{\sigma}\otimes ...\otimes b_{x_{n-1}}^{\sigma},$$ where for $j\in X_{n-1}$ we define 
$$b_j^{\sigma}=\prod_{\{s\in X_{n} | d_i (s)=j\}}a_{\sigma(s)}.$$
In the last formula the product is ordered over $s$. Notice that the order that we pick on $Z_n^i(j)$ is not important, we just want to make sure that we cover all the possible ordered products. 

As one expects, if $A$ is commutative and $M$ is a symmetric $A$-bimodule we get the usual higher  order Hochschild homology $H^{\mathbf{X}}_{\bullet}(A,M)$. 
\end{example}

\begin{example}
 Take $A$ a commutative $k$-algebra, and $M$ a symmetric $A$-bimodule.  Take $e\in M_l(A)$, and  $m\in M_l(B)$ such that $e^2=e$,  and $em=me=m$. Consider the element 
$$W_{n}(e,m)=m\otimes e^{\otimes x_n}\in C^{X_\bullet}_n(A,M).$$
Notice that $d_i^{\alpha}(W_n(e,m))=W_{n-1}(e,m)$, which means that if we define $C_n=kW_n(e,m)$ we get a $\lambda$-subcomplex, and so $W_n(e,m)\in \Theta(\mathcal{C}^{X_\bullet}(A,M))_n$. Moreover $\partial_{2n}(W_{2n}(e,m))=0$ and so we get an element $\overline{W_{2n}(e,m)}\in H_{2n}(M_l(A),M_l(M))$. 
\end{example}

\begin{remark}
Notice that  the $\mathbf{\Lambda}$-system from Example \ref{example3} is completely determined by $A$, $M$ and the simplicial set $\mathbf{X}$.  
We  denote the homology groups $H_n(\Theta(\mathcal{C}^{\mathbf{X}}(A,M)))$  by   $H_n^{\mathbf{X}}(A,M)$. When $A$ is commutative and $M$ is a symmetric $A$-bimodule,  we recover the higher order Hochschild homology, so this notation is consistent with  \cite{p}.  When $\mathbf{X}=S^2$ with the usual simplicial structure (see \cite{l2}), we recover Example \ref{example2}.\label{remark8}
 
\end{remark}

\subsection{Secondary Hochschild Homology}

The next example is associated with the secondary Hochs\-child homology $HH_{\bullet}(A,B,\varepsilon)$. Recall that in \cite{jma} we need $A$ to be a $B$-algebra, and in particular $B$ must be commutative. Using the construction from the previous section, we are able to drop that condition.

\begin{example}\label{triple}
Let $A$ and $B$ be $k$-algebras, and $\varepsilon:B\rightarrow A$ be a $k$-algebra morphism.  Here we do not assume $B$ is commutative.
Take $\mathbf{\Lambda}=\{\Lambda_n^i\}$ as follows: for $0\leq i\leq n-1$
$$\Lambda_n^i=\{\sigma=(\sigma_0,\sigma_1,\ldots,\sigma_i,\ldots,\sigma_{n-1})~:~\sigma_j\in\{1,\tau\} \text{~for~} j\neq i,~\sigma_i\in\{l,c,r\}\},$$
and 
$$\Lambda_n^n=\{\sigma=(\sigma_0,\sigma_1,\ldots,\sigma_i,\ldots,\sigma_{n-1})~:~\sigma_j\in\{1,\tau\} \text{~for~} j\neq 0,~\sigma_0\in\{l,c,r\}\}.$$
We define a $\mathbf{\Lambda}$-system $\mathcal{E}(A,B,\varepsilon)$ where we set $$E_n(A,B,\varepsilon)=A^{\otimes n+1}\otimes B^{\otimes\frac{n(n+1)}{2}}.$$
For $0\leq i\leq n-1$ and $\sigma\in \Lambda_n^i$ define $d_i^{\sigma}:A^{\otimes n+1}\otimes B^{\otimes\frac{n(n+1)}{2}}\rightarrow A^{\otimes n}\otimes B^{\otimes\frac{n(n-1)}{2}}$  given by
$$
d_i^{\sigma}\left(\otimes\begin{pmatrix}
a_0 & b_{0,1} & \cdots & b_{0,n-1} & b_{0,n}\\
 & a_1 & \cdots & b_{1,n-1} & b_{1,n}\\
 &  & \ddots & \vdots & \vdots\\
 &  &  & a_{n-1} & b_{n-1,n}\\
 &  &  &  & a_{n}\\
\end{pmatrix}\right)=
$$
$$
\otimes\begin{pmatrix}
a_0 & \cdots & b_{0,i-1} & \sigma_0(b_{0,i}\otimes b_{0,i+1}) & b_{0,i+2} & \cdots & b_{0,n}\\
 & \ddots & \vdots & \vdots & \vdots & \ddots & \vdots\\
  &  & a_{i-1} & \sigma_{i-1}(b_{i-1,i}\otimes b_{i-1,i+1}) & b_{i-1,i+2} & \cdots & b_{i-1,n}\\
  &  &   & \sigma_i(a_i\otimes a_{i+1}\otimes b_{i,i+1}) & \sigma_{i+1}(b_{i,i+2}\otimes b_{i+1,i+2}) & \cdots & \sigma_{n-1}(b_{i,n}\otimes b_{i+1,n})\\
  &  &  &  & a_{i+2} & \cdots & b_{i+2,n}\\
  &  &  &  &  & \ddots & \vdots\\
  &  &  &  &  &  & a_{n}\\
\end{pmatrix}.
$$
Where, for $j\neq i$ we have
$$
\sigma_j(b_1\otimes b_2)=
\begin{cases}
b_1b_2 & \text{if~~} \sigma_j=1\in\Lambda_n^i\\
b_2b_1 & \text{if~~} \sigma_j=\tau\in\Lambda_n^i\\
\end{cases}
$$
for all $b_1,b_2\in B$, and
$$
\sigma_i(a_1\otimes a_2\otimes b)=
\begin{cases}
\varepsilon(b)a_1a_2 & \text{if~~} \sigma_i=l\in\Lambda_n^i\\
a_1\varepsilon(b)a_2 & \text{if~~} \sigma_i=c\in\Lambda_n^i\\
a_1a_2\varepsilon(b) & \text{if~~} \sigma_i=r\in\Lambda_n^i\\
\end{cases}
$$
for all $a_1,a_2\in A$ and $b\in B$.
Finally, for $i=n$ we have
$$
d_n^{\sigma}\left(\otimes\begin{pmatrix}
a_0 & b_{0,1} & \cdots & b_{0,n-1} & b_{0,n}\\
  & a_1 & \cdots & b_{1,n-1} & b_{1,n}\\
  &   & \ddots & \vdots & \vdots\\
  &   &  & a_{n-1} & b_{n-1,n}\\
  &   &  &   & a_{n}\\
\end{pmatrix}\right)=
$$
$$
\otimes\begin{pmatrix}
\sigma_0(a_n\otimes a_0\otimes b_{0,n}) & \sigma_1(b_{0,1}\otimes b_{1,n}) & \cdots & \sigma_{n-2}(b_{0,n-2}\otimes b_{n-2,n}) & \sigma_{n-1}(b_{0,n-1}\otimes b_{n-1,n})\\
  & a_1 & \cdots & b_{1,n-2} & b_{1,n-1}\\
  &  & \ddots & \vdots & \vdots\\
  &  &  & a_{n-2} & b_{n-2,n-1}\\
  &  &  &   & a_{n-1}\\
\end{pmatrix}.
$$
\end{example}

\begin{remark} We denote the homology of $\Theta(\mathcal{E}(A,B,\varepsilon))$ by $HH_{\bullet}(A,B,\varepsilon)$ and call it the secondary homology of the triple $(A,B, \varepsilon)$. 
Notice that  if $B$ is commutative and $\varepsilon(B)\subseteq \mathcal{Z}(A)$, we recover the usual  secondary Hochschild homology $HH_*(A,B,\varepsilon)$ as defined in \cite{jma}. 
\end{remark}

\begin{example}
Take $B$ to be a commutative $k$-algebra, $A$ to be a $k$-algebra, $\varepsilon:B\to A$ to be a morphism of $k$-algebras such that $\varepsilon(B)\subseteq \mathcal{Z}(A)$, and $\iota:M_l(B)\to M_l(A)$ to be the induced $k$-algebra morphism. Take $e\in M_l(A)$, and  $f\in M_l(B)$ such that $e^2=e$, $f^2=f$, and $ef=fe=e$. Consider the element 
$$T_{n}(e,f)=\otimes\begin{pmatrix}
e & f & \cdots & f & f\\
  & e & \cdots & f & f\\
  &  & \ddots & \vdots & \vdots\\
  &  &  & e & f\\
  &  &  &   & e\\
\end{pmatrix} \in E_{n}(M_l(A),M_l(B), \iota).
$$
Notice that $d_i^{\alpha}(T_n(e,f))=T_{n-1}(e,f)$. This means that if we define $C_n=kT_n(e,f)$, we get a $\lambda$-subcomplex. In particular $T_n(e,f)\in \Theta(\mathcal{E}(M_l(A),M_l(B), \iota))$. Moreover $\partial_{2n}(T_{2n}(e,f))=0$ and so we get an element $\overline{T_{2n}(e,f)}\in H_{2n}(M_l(A),M_l(B), \iota)$. 
\end{example}


\section{Back to the Hochschild Homology} 
In this section $A$ is a commutative $k$-algebra, and $M$ is a symmetric $A$-bimodule. For the matrix algebra $M_l(A)$ we have two possible different ways of defining Hochschild homology. We have the classical $H_n(M_l(A),M_l(M))$ (as in the preliminary section), and $H^{S^1}_n(M_l(A),M_l(M))$ (as in the previous section). We will show that the two constructions agree.

We recall the simplicial structure on $S^1$. Take $X_0=\{*_0\}$, and $X_n=\{*_n\}\cup\{I^a_b\; \vert \; a+b+1=n\}$ with 
\begin{eqnarray}
d_i(*_{n})=*_{n-1},\label{A1}
\end{eqnarray} 
\begin{eqnarray}
d_i(I^a_b)= \left\{\begin{array}{ll}
  *_{a+b}& \mbox{ if $a=0$ and $i=0$}\\ 
  I^{a-1}_b & \mbox{ if  $a\neq 0$ and $i\leq a$ }\\
 I^a_{b-1}& \mbox{ if $b\neq 0$ and $i>a$}\\
 *_{a+b}& \mbox{ if $b=0$ and $i=n=a+1$,}
 \end{array}\right.\label{A2}
\end{eqnarray}
\begin{eqnarray}
s_i(*_{n})=*_{n+1},\label{B1}
\end{eqnarray} 
\begin{eqnarray}
s_i(I^a_b)= \left\{\begin{array}{ll}
  I^{a+1}_b & \mbox{ if   $i\leq a$ }\\
 I^a_{b+1}& \mbox{ if  $i>a$.}
 \end{array}\right.\label{B2}
\end{eqnarray}
Next we give the details for the $\Delta$-indexing set $\mathbf{\Lambda}^{S^1}$, and the $\mathbf{\Lambda}^{S^1}$-system $\mathcal{C}^{S^1}(M_l(A),M_l(M))$  as described in Example \ref{example3}. 

One can see that  $\vert {\mathbf{\Lambda}}_n^i\vert =2$, so we can identify ${\mathbf{\Lambda}}_n^i$ with the set $\{1,\tau\}$. For all $n\in \mathbb{N}$ we have  $$C^{S^1}_n(M_l(A),M_l(M))=M_l(M)\otimes M_l(A)^{\otimes n}.$$
For $0\leq i\leq n$ and $\alpha\in \mathbf{\Lambda}^i_n$, we have $\delta_i^{\alpha}:C^{S^1}_n(M_l(A),M_l(M))\to C^{S^1}_{n-1}(M_l(A),M_l(M))$ determined by
\begin{eqnarray}
\delta_i^1(x_0\otimes x_1\otimes ...\otimes x_n)= \left\{\begin{array}{ll}
  x_0 x_1\otimes x_2\otimes ...\otimes x_n& \mbox{ if $i=0$},\\ 
  x_0\otimes ...\otimes x_{i-1}\otimes x_ix_{i+1}\otimes x_{i+2}\otimes ...\otimes x_n & \mbox{ if  $1\leq i\leq n-1$},\\
 x_nx_0\otimes x_1\otimes ...\otimes x_{n-1} & \mbox{ if  $i= n$},\\
 \end{array}\right.\label{D1}
\end{eqnarray}
and
\begin{eqnarray}
\delta_i^{\tau}(x_0\otimes x_1\otimes ...\otimes x_n)= \left\{\begin{array}{ll}
  x_1 x_0\otimes x_2\otimes ...\otimes x_n& \mbox{ if $i=0$},\\ 
  x_0\otimes ...\otimes x_{i-1}\otimes x_{i+1}x_{i}\otimes x_{i+2}\otimes ...\otimes x_n & \mbox{ if  $1\leq i\leq n-1$},\\
 x_0x_n\otimes x_1\otimes ...\otimes x_{n-1} & \mbox{ if  $i= n$}.\\
 \end{array}\right.\label{Dt}
\end{eqnarray}

We are now ready to state the following result. 

\begin{proposition} Let $A$ be a commutative $k$-algebra and $M$ a symmetric $A$-bimodule. Then we have
$$H^{S^1}_n(A,M)\simeq H_n(A,M)\simeq H_n(M_l(A),M_l(M))\simeq H_n^{S^1}(M_l(A),M_l(M)).$$
\end{proposition}
\begin{proof}
Since $A$ is commutative and $M$ is symmetric, the first isomorphism  is known from \cite{p}. Also, it is well known  from \cite{lo} that the maps $i_A:A\to M_l(A)$ and $i_M:M\to M_l(M)$ determined by
$$x\to \begin{pmatrix} 
x & 0 & \cdots & 0\\
0& 0 & \cdots & 0\\
\vdots& \vdots & \cdots & \vdots\\
0&0 & \cdots & 0
\end{pmatrix}$$
can be extended to a morphism of pre-simplicial modules 
$$\iota_2: (C_n(A, M),d_i) \to (C_n(M_l(A), M_l(M)),d_i),$$ 
which induced an isomorphism at the level of homology. Thus Hochschild homology is Morita invariant, that is $H_n(A,M)\simeq H_n(M_l(A),M_l(M))$. 
 
Take $\mathbf{\Upsilon}$ the trivial  $\Delta$-indexing set (i.e. $\vert\Upsilon_n^i\vert =1$ for all $n$ and all $0\leq i\leq n$). Then every pre-simplicial module is a $\mathbf{\Upsilon}$-system. In particular $\mathcal{C}(A,M)$ and $\mathcal{C}(M_l(A),M_l(M))$ are the ${\bf \Upsilon}$-systems associated to $(C_n(A, M),d_i)$ and $(C_n(M_l(A), M_l(M)),d_i)$, respectively. 

Since $A$ is commutative and $M$ is $A$-symmetric, the maps $i_A:A\to M_l(A)$ and $i_M:M\to M_l(M)$  induce  a $\lambda$-morphism  $\mathcal{C}(A,M)\to \mathcal{C}^{S^1}(M_l(A),M_l(M))$ (as in Definition \ref{def81}). By Lemma \ref{lemma10} we obtain a morphism of pre-simplicial modules 
$$\iota_0:(C_n(A,M), d_i)\to (\Theta(\mathcal{C}^{S^1}(M_l(A),M_l(M)))_n, \delta_i).$$
Also, it easy to check that the identity map $\mathcal{C}^{S^1}(M_l(A),M_l(M)) \to \mathcal{C}((M_l(A),M_l(M)))$ is a $\lambda$-morphism  (as in Definition \ref{def81}). Again by Lemma \ref{lemma10} this gives a morphism of pre-simplicial $k$-modules 
$$\iota_1: (\Theta(\mathcal{C}^{S^1}(M_l(A),M_l(M)))_n, \delta_i) \to (C_n((M_l(A),M_l(M))), d_i).$$
Finally, we have that $\iota_2=\iota_1\iota_0$, and since $\iota_2$ induces an isomorphism in homology we get that $\iota_1$ will also induce an isomorphism $H_n^{S^1}(M_l(A),M_l(M))\simeq H_n(M_l(A),M_l(M))$, which finishes the proof. 
\end{proof}

\section{Final Remarks}

The setting in Theorem \ref{max} is quite general, and if applied to poorly chosen ${\mathbf \Lambda}$-systems, the theorem is not likely to give interesting results. One has to balance between $\Delta$-indexing sets that are too big or too small. 

The results from the previous section show that when $\mathbf{X}_\bullet$ is modeled by $S^1$, our construction of higher order Hochschild homology for noncommutative algebras behaves as one would hope. However, the proof depends heavily on the already known existence and properties of Hochschild homology for noncommutative algebras. 

If $A$ is a commutative $k$-algebra, $M$ a symmetric $A$-bimodule, and $\mathbf{X}_{\bullet}$ a simplicial set one can show that we have a morphism $H_n^{\bf X_{\bullet}}(A,M)\to H_n^{\bf X_{\bullet}}(M_l(A),M_l(M))$. It would be interesting  to prove that this morphism is actually an isomorphism (i.e we have  Morita invariance). 

One can easily check the functoriality of $H^{X_{\bullet}}(A,M)$. It would be interesting to see if the construction   of $H_{\bullet}(\mathcal{M}^{\mathbf{X}}(A,M))$ is invariant under the homotopy equivalence of the simplicial set $\mathbf{X}$. Notice that we did not use the degeneracy maps of the simplicial set $\mathbf{X}$, but that information could be easily incorporated in some variation of Theorem  \ref{max} (that would deal with maximal simplicial modules instead of maximal pre-simplicial modules). 

Similar constructions can be done if one wants to define higher order Hochschild cohomology, or for the generalized higher Hochschild (co)homology  (see \cite{bm} or \cite{hhl}).

\bibliographystyle{amsalpha}

\begin{thebibliography}{A}


\bibitem
{b}
B. R. Corrigan-Salter, \textit{Higher Order Hochschild (Co)homology of Noncommutative Algebras},
Bull. Belg. Math. Soc. Simon Stevin, {\bf 25} (2018), 741--754.


\bibitem
{bm}
B. R. Corrigan-Salter, and M. D. Staic, \textit{ Higher Order and Secondary Hochschild Cohomology},  C. R. Math. Acad. Sci. Paris, {\bf 354} (2016), no. 11, 1049--1054. 


\bibitem
{g1}
M. Gerstenhaber, \textit{The Cohomology Structure of an Associative Ring}, Ann. of Math. (2) {\bf 78} (1963), 267--288.

\bibitem
{g2}
M. Gerstenhaber, \textit{On the Deformation of Rings and Algebras}, Ann. of Math. (2)
{\bf 79} (1964), 57--103.



\bibitem
{gi}
G. Ginot, \textit{Higher order Hochschild Cohomology}, C. R. Math. Acad. Sci. Paris, {\bf 346} (2008), 5--10.


\bibitem
{gi2}
G. Ginot, \textit{Hodge filtration and operations in higher Hochschild (co)homology and
applications to higher string topology}, in Algebraic topology,  Springer International Publishing, Lecture Notes in Mathematics, {\bf 2194} (2017), 1--104.


\bibitem
{gtz}
G. Ginot, T. Tradler, M. Zeinalian, \textit{Higher Hochschild homology, topological chiral homology and factorization algebras}, Commun. Math. Phys. {\bf 326} (2014), 635–-686.

\bibitem
{hhl}
G. Halliwell, E. Honing, A. Lindenstrauss, B. Richter, and I. Zakharevich, \textit{Relative Loday constructions and applications to higher THH-calculations},  Topology Appl., {\bf 235} (2018), 523–-545.


\bibitem
{h} G. Hochschild, \textit{On the Cohomology Groups of an Associative Algebra}, Ann. of Math. (2) {\bf 46} (1945), 58--67.



\bibitem
{l2}
J. Laubacher, \textit{Secondary Hochschild and Cyclic (Co)homologies},  Ph.D thesis, (2017).

\bibitem
{jma}
J. Laubacher, M. D. Staic, and A. Stancu, \textit{Bar simplicial modules and secondary
cyclic (co)homology}. J. Noncommut. Geom., {\bf 12} (2018),  865--887.


\bibitem
{lo} J. L. Loday, \textit{Cyclic Homology}, Springer-Verlag, Grundlehren der mathematischen Wissenschaften, {\bf 301} (1992).


\bibitem
{p} T. Pirashvili, \textit{Hodge decomposition for higher order Hochschild homology}, Ann. Sci. Ecole Norm. Sup., (4) {\bf 33} (2000), 151--179.



\bibitem
{sta} M. D. Staic, \textit{Secondary Hochschild Cohomology}, Algebras and Representation Theory,  {\bf 19} Issue 1 (2016), pp 47--56.



\end{thebibliography}

\end{document}